\documentclass[12pt]{amsart}
\usepackage{amsmath,amsthm,amsfonts,amssymb,times,latexsym,mathabx,url}

\newtheorem{theorem}{Theorem}[section]

\newtheorem{lem}[theorem]{Lemma}
\newtheorem{cor}[theorem]{Corollary}

\newtheorem{rem}[theorem]{Remark}
\numberwithin{equation}{section}

\renewcommand{\a}{\alpha}
\renewcommand{\b}{\beta}
\newcommand{\g}{\gamma}
\renewcommand{\d}{\delta}
\newcommand{\e}{\varepsilon}
\renewcommand{\l}{\lambda}
\renewcommand{\o}{\omega}
\renewcommand{\O}{\Omega}

\newcommand{\A}{\mathcal{A}}
\newcommand{\E}{\mathbb{E}}

\renewcommand{\L}{\mathcal{L}}
\newcommand{\N}{\mathbb{N}}
\renewcommand{\P}{\mathbb{P}}
\newcommand{\R}{{\mathbb R}}
\newcommand{\s}{\mathfrak{S}}
\newcommand{\T}{{\mathbb T}}
\newcommand{\Z}{\mathbb{Z}}

\renewcommand{\leq}{\leqslant}
\renewcommand{\geq}{\geqslant}

\begin{document}

\title{The stochasticity parameter of quadratic residues}
\author{Mikhail R. Gabdullin} 
\date{}
\address{Steklov Mathematical Institute,
Gubkina str., 8, Moscow, Russia, 119991}
\email{gabdullin.mikhail@yandex.ru, gabdullin@mi-ras.ru} 

\begin{abstract}
Following V. I. Arnold, we define the stochasticity parameter $S(U)$ of a subset $U$ of $\Z/M\Z$ to be the sum of squares of the consecutive distances between elements of $U$. In this paper we study the stochasticity parameter of the set $R_M$ of quadratic residues modulo $M$. We present a method which allows to find the asymptotics of $S(R_M)$ for a set of $M$ of positive density. In particular, we obtain the following two corollaries. Denote by $s(k)=s(k,\Z/M\Z)$ the average value of $S(U)$ over all subsets $U\subseteq \Z/M\Z$ of size $k$, which can be thought of as the stochasticity parameter of a random set of size $k$. Let $\s(R_M)=S(R_M)/s(|R_M|)$. We show that 

a) $\varliminf_{M\to\infty}\s(R_M)<1<\varlimsup_{M\to\infty}\s(R_M)$;

\smallskip

b) the set $\{ M\in \N: \s(R_M)<1 \}$ has positive lower density. 
\end{abstract}

\date{\today}
\maketitle

\section{Introduction}

For a positive integer $n\geq2$, let $U=\{0\leq u_1<u_2<...<u_k<n \}$ be a subset of size $k\geq2$ of the circle $\T_n=\R/n\Z$ or of the cyclic group $\Z/n\Z$. For $i=1,...,k-1$, we denote by $s_i=u_{i+1}-u_i\in\R^+$ the consecutive distances between elements of $U$, and write $s_k=u_1+n-u_k\in\R^+$. Following Arnold \cite[\S9]{Arn}, we define the stochasticity parameter $S(U)$ of the set $U$ by
$$
S(U)=\sum_{i=1}^{k}s_i^2.
$$
While we study this parameter only in the discrete case of $\Z/n\Z$, we would like to  focus for a moment on the continuous case of $\T_n$, which is useful for some heuristic in what follows. Since $\frac{n}{k} =\frac{1}{k}\sum_{i=1}^ks_i\leq \left(\frac{1}{k}\sum_{i=1}^ks_i^2\right)^{1/2}$ and $\sum_{i=1}^ks_i^2 < \left(\sum_{i=1}^ks_i\right)^2=n^2$, we see that 
$$ \min_{|U|=k} S(U)= \frac{n^2}{k}
$$
(the minimum is attained when all $s_i$ are equal to $n/k$) and
$$\sup_{|U|=k} S(U)=n^2
$$
($S(U)$ is close to this value when $U$ is contained in a small interval). Now let $U$ be a random subset of $\T_n$ of size $k$ (we think of it as a random point of $\{(s_1,...,s_k)\in[0,n]^k: \sum_{i=1}^ks_i=n\}$ endowed with the normalized Lebesgue measure).  
It is not hard to show that $\P\left(s_i>t\right)=\P\left(s_1>t\right) = \left(\frac{n-t}{n}\right)^{k-1}$ for each $i$ and $t\in(0,n)$, and we thus have
$$
\E s_i=\int_0^n\P(s_i>t)dt=n/k, \quad \E s_i^2=\int_0^n2t\,\P(s_i>t)dt=\frac{2n^2}{k(k+1)}\,.
$$ 
It follows that
\begin{equation}\label{1.1}
\E S(U)=\frac{2n^2}{k+1}
\end{equation}
and
\begin{equation}\label{1.2}
\P(s_i>t\E s_i) = (1-t/k)^{k-1} =  e^{-t} (1+o(1)), \quad k\to\infty,	
\end{equation}
uniformly for $0\leq t \leq t_0$ for any fixed $t_0>0$, and so for large $k$ the normalized gaps $s_i/\E s_i$ have the exponential distribution with parameter $1$. 
 
Now let $U$ be a subset of a cyclic group $\Z/n\Z$ of a fixed size $k$. Again, it is easy to show that $S(U)$ is minimal when the points of $U$ are nearly equidistributed, and is maximal when $U$ is an interval; thus, a too small or a too large value of $S(U)$ is indicative of nonrandom behavior of $U$. Having noted these extreme cases, we are now interested in comparing $S(U)$ with the average $s(k,\Z/n\Z)$ of the stochasticity parameter taken over all subsets of $\Z/n\Z$ of size $k$, which can be interpreted as the stochasticity parameter of a random set of size $k$. In what follows we write $s(k)$ instead of $s(k,\Z/n\Z)$ for brevity, since the ambient group will always be clear from context; we also mention here the discrete analog of (\ref{1.1}):
\begin{equation}\label{1.3} 
s(k)=\frac{n(2n-k+1)}{k+1};
\end{equation}
note that this agrees with (\ref{1.1}) in the sense that  $s(k)\sim \frac{2n^2}{k+1}$ whenever $k=o(n)$.  

It is worth mentioning that several classical papers were devoted to the stochasticity parameter of the set $\mathfrak{A}_n=\{1=a_1<...<a_{\varphi(n)}=n-1: a_i \mbox{ coprime to } n \}$ of reduced residues modulo $n$. Firstly, Erd\H{o}s \cite{Erd} conjectured that
$$
S(\mathfrak{A}_n) \ll \frac{n^2}{\varphi(n)}
$$
(note that the right-hand side is of order $s(|\mathfrak{A}_n|)$). A more general question was studied by Hooley \cite{Hoo1} (see also \cite{Hoo2, Hoo3}), who showed that 
\begin{equation}\label{1.4}
\sum_{i=1}^{\varphi(n)}(a_{i+1}-a_i)^{\l} \ll n^{\l}(\varphi(n))^{1-\l}
\end{equation}
for $0<\l<2$, and proved the weaker result 
$$
S(\mathfrak{A}_n) \ll n\log\log n
$$
in the case $\l=2$. Montgomery and Vaughan \cite[Corollary 1]{MV} established (\ref{1.4}) for all $\l>0$, succeeding in proving Erd\H{o}s' conjecture. Their key ingredient (and actually their main theorem) is the bound 
\begin{equation}\label{1.5}
\mathcal{X}_q(\mathfrak{A}_n; h):= \sum_{l=1}^n\left(\sum_{\substack{m=1 \\(l+m,n)=1}}^h 1 - h\frac{\varphi(n)}{n}\right)^q	\ll n\left(\frac{h\varphi(n)}{n}\right)^{q/2}+h\varphi(n),
\end{equation}
where $q$ is a fixed positive integer. The mentioned result of Hooley also relies on the corresponding bound for $\mathcal{X}_2(\mathfrak{A}_n;h)$, which is equal to the variance of the random variable
$$
\#\{1\leq m\leq h: (l+m,n)=1 \}
$$
(with fixed $h$ and $l$ drawn uniformly at random from $\Z/n\Z$).

In this work we study the stochasticity parameter of the set 
$$
R_M=\{a^2: a\in\Z/M\Z \}
$$ 
of quadratic residues modulo $M$. It has multiplicative structure, and, as one may expect, there are much evidence for its random behavior as an additive set: say, in the case of prime modulus $p$, the probability that a random element of $\Z/p\Z$ is a quadratic residue is close to $1/2$, and it can be shown (see Lemma \ref{lem2.2} below) that for any arbitrary disjoint subsets $C_1$ and $C_2$ of $\Z/p\Z$ the events $x+c_1\in R_p$ and $x+c_2\notin R_p$, $c_1\in C_1$, $c_2\in C_2$, are nearly independent, provided that $|C_1|+|C_2|\ll \log p$ (in particular, $R_p$ contains any additive configurations of length $\ll\log p$). The last phenomenon also comes up in the case of square-free moduli; see Lemma \ref{lem2.3}. In the general case, the main evidence for random behavior of $R_M$ is the following. Write $R_M=\{0=r_1<r_2<...<r_{|R_M|}\}$, set $r_{|R_M|+1}=M$, and take an index $i$ randomly and uniformly in $\{1,...,|R_M|\}$; then on average we have $\E(r_{i+1}-r_i)=\frac{M}{|R_M|}$. The limit distribution of the distances $r_{i+1}-r_i$ was found by Kurlberg and Rudnick  (see \cite{KR} and \cite{Kur}); they showed that (here and in what follows $\o(M)$ stands for the number of prime divisors of $M$)
\begin{equation}\label{1.6} 
\P\left(r_{i+1}-r_i>t\frac{M}{|R_M|}\right)=e^{-t}(1+o(1)), \qquad  \o(M)\to\infty,
\end{equation}
uniformly in the range $0\leq t\leq t_0$ for any fixed $t_0$, which is exactly the distribution (\ref{1.2}) of gaps for a random set. This result supports the conjecture that, as $\o(M)\to\infty$,
\begin{multline*} 
S(R_M)=\frac{M^2}{|R_M|}\E\left(\frac{r_{i+1}-r_i}{M/|R_M|}\right)^2 \\
=\frac{M^2}{|R_M|}\int_0^{\infty}2t\,\P\left(r_{i+1}-r_i>t\frac{M}{|R_M|}\right)dt
\sim \frac{M^2}{|R_M|}\int_0^{\infty}2te^{-t}dt=\frac{2M^2}{|R_M|};
\end{multline*}
however, to make this rigorous we need good upper bounds for the contribution of large gaps between residues. On the other hand, if $\o(M)\to\infty$, then $|R_M|\to\infty$ and $M/|R_M|\to\infty$, and we have by (\ref{1.3}) 
$$
s(|R_M|)=M\frac{2M-|R_M|+1}{|R_M|+1}\sim\frac{2M^2}{|R_M|}, \quad \o(M)\to\infty,
$$	
as one may expect. It will often be convenient for us to describe the behavior of $S(R_M)$ in terms of the quantity 
\begin{equation*}
\s(R_M)=\frac{S(R_M)}{s(|R_M|)}.	
\end{equation*}
We note that $s(|R_M|)$ turns out to be a better approximation for $S(R_M)$ than $\frac{2M^2}{|R_M|}$, since it also covers the case of a prime modulus: Garaev, Konyagin, and Malykhin \cite{GKM} proved that 
\begin{equation}\label{1.7} 
\s(R_p)= 1+o(1), \quad p\to\infty,
\end{equation}
but it is easy to compute that $s(|R_p|)=p\frac{2p-(p+1)/2+1}{(p+1)/2+1}\sim 3p$, whereas $\frac{2p^2}{|R_p|}=\frac{2p^2}{(p+1)/2}\sim 4p$ (and so $S(R_M)$ is not asymptotically equivalent to $\frac{2M^2}{|R_M|}$ in general).

We present a method which allows us to find the asymptotics of $S(R_M)$ for a set of $M$ of positive density\footnote{Here and in what follows we consider the asymptotic density of a set $A\subseteq\N$, which is defined as $\lim_{N\to\infty}\frac{|A\cap\{1,\ldots,N\}|}{N}$. The lower asymptotic density (which appears only in Corollary \ref{cor1.3}) is defined similarly but with $\varliminf$ instead of $\lim$.}. Let $C_0>2$ be a sufficiently large absolute constant and let $P^-(n)$ denote the smallest prime divisor of $n$. We consider the set $\O$ of sufficiently large positive integers $M$ such that

\begin{itemize}
\item[(i)] $M$ is square-free;
	
\item[(ii)] $0.9\log\log M < \o(M) < 1.1\log\log M$;	
	
\item[(iii)] $M$ has the representation of the form $M=Am$, where $3\leq A \leq (\log M)^{0.1}$ and $P^-(m)>(\log M)^{C_0}$ (note that this representation is unique).
\end{itemize}
Standard sieve methods imply that the density of the set $\O$ is equal to $e^{-\g}\frac{3}{5\pi^2C_0}$, where $\gamma$ is the Euler constant; we will prove this in Section \ref{Sec6}. 	

Our main result is the following.

\begin{theorem}\label{th1.1}
For $M\in \O$, we have
$$
S(R_M)= \frac{2M^2}{|R_M|} - MA|R_A|^{-1}+O(M(\log M)^{-0.4}).
$$
\end{theorem}

Note that (\ref{1.3}) implies $M^{-1}s(k)=\frac{2M-k+1}{k+1}=\frac{2(M+1)}{k+1}-1=
\frac{2M}{k}-1+O\left(\frac{M}{k^2}\right)$.
Applying this to $M\in\O$, we obtain
\begin{equation*}
s(|R_M|)=\frac{2M^2}{|R_M|}-M+O((\log M)^2).
\end{equation*}
Since $M|R_M|^{-1}$ is close to $A|R_A|^{-1}2^{\o(M)}$ and $|R_A|<A$ for any $A\geq3$,  Theorem \ref{th1.1} implies the following two corollaries, the first of which extends (\ref{1.7}).

\begin{cor}\label{cor1.2} 
For $M\in\O$, we have
$$
\s(R_M)=1+o(1), \quad M\to\infty.
$$	
\end{cor}

\begin{cor}\label{cor1.3}
The set $\{ M\in \N: \s(R_M)<1 \}$ has positive lower density.
\end{cor}

Let us denote by $\delta_0$ the lower density of the set $\{M\in\N: \s(R_M)<1\}$; it follows that $\d_0\geq e^{-\g}\frac{3}{5\pi^2C_0}$. One could extract the exact value of $C_0$ from our argument and thus obtain an explicit lower bound for $\delta_0$, but this bound would be very poor (mainly because of the estimates from Section \ref{Sec4}). It seems to be hard even to show that $\delta_0\geq 0.1$ using our method; thus, we did not try to optimize our lower bounds for $\d_0$ and keep the paper relatively short. We also note that, while proving Theorem \ref{th1.1}, one can probably avoid the condition (i) in the definition of the set $\O$, but it simplifies our argument and added for convenience. Further, due to Hardy and Ramanujan \cite{HR}, (ii) holds for almost all integers (that is, for a set of density $1$), and thus we can assume it without affecting the density of $\O$. The third condition (iii) is crucial for us: we will need it for making some error terms small, and anyway it occurs in different steps of our proof. Finally, we note that our method seems to be incapable of proving that the inequality $\s(R_M)>1$ holds for a large set of $M$; nevertheless, we believe that the set of such $M$ also has positive density.

In the proof of Theorem \ref{th1.1}, we rewrite $S(R_M)$ in terms of the following function $f_A(y)$. Let $s_1,...,s_{|R_A|}$ be the distances between consecutive quadratic residues modulo $A$; we consider the indices of $s_i$ as elements of $\Z/|R_A|\Z$ rather than integers. For $k=0,...,|R_A|-1$, we set 
$\a_k=\sum_{i=1}^{|R_A|}\left(\sum_{l=i}^{i+k}s_l\right)^2$. In particular,  $\a_0=\sum_{i=1}^{|R_A|}s_i^2=S(R_A)$ and $\a_{|R_A|-1}=A^2|R_A|$. Now we define
$$
f_A(y)=\frac{F_A(y)}{Q_A(y)}, 
$$
where $Q_A(y)=1+y+\ldots+y^{|R_A|-1}$ and $F_A(y)=\sum_{k=0}^{|R_A|}\b_ky^k$ is the reciprocal polynomial with the coefficients $\b_0=\b_{|R_A|}=\sum_is_i^2$ and $\b_k=2\sum_{i=1}^{|R_A|}s_is_{i+k}$ 
for $0<k<|R_A|$. 

Another situation in which our method gives a good approximation of $S(R_M)$ is the following ``perturbed'' version of (\ref{1.7}). We consider the moduli of the form $M=Ap$, where $p$ is a prime and $A$ is small enough in terms of $p$; it turns out that the asymptotic behavior of $S(R_M)$ depends substantially on this additional factor. 

\begin{theorem}\label{th1.4}
Let $M=Ap$, where $p$ is a prime and $A$ is square-free with $A\leq p^{0.01}$. Then
$$
S(R_M)=2f_A(0.5)p + O(A^4p^{0.95}).
$$
\end{theorem}

It can be easily seen that $f_A(0.5)\gg F_A(0.5) \geq \b_0=S(R_A) \geq A$, and thus the error term in Theorem \ref{th1.4} is small enough.

One may wonder whether (\ref{1.7}) is still true if we consider all positive integers instead of primes. Theorem \ref{th1.4} gives the negative answer for this question. Indeed, using (\ref{1.3}) for the same $M$, we find  
\begin{equation*}\label{rand2}
s(|R_M|)=\left(\frac{4A^2}{|R_A|}-A\right)p+O(A^2|R_A|^{-1});
\end{equation*}
thus, it is enough to compare $2f_A(0.5)$ and $\frac{4A^2}{|R_A|}-A$ and find two appropriate values of $A$ which give us both possible inequalities. Numerical computations\footnote{See https://github.com/fractalon/arXiv2010.04982\_appendix. The author is grateful to Konstantin Olmezov for this code.} show that $2f_A(0.5)<\frac{4A^2}{|R_A|}-A$ for all $3\leq A\leq 100$, $A\neq89$, and
$2f_A(0.5)>\frac{4A^2}{|R_A|}-A$ for $A=89$. This yields 

\begin{cor}\label{cor1.5}
We have
$$
\varliminf_{M\to\infty}\s(R_M)<1<\varlimsup_{M\to\infty}\s(R_M).
$$	
\end{cor}

We stress that in the above situation $\o(M)$ remains bounded. However, as was mentioned after (\ref{1.6}), it is natural to conjecture that
$$ 
\s(R_M)=1+o(1), \quad \o(M)\to\infty. 
$$
Note that the trivial lower bound $S(R_M)\geq\frac{M^2}{|R_M|}$ (obtained from Cauchy-Schwarz), together with (\ref{1.3}), implies that $\varliminf_{M\to\infty}\s(R_M)\geq0.5$; on the other hand, we do not know if $\varlimsup_{M\to\infty}\s(R_M)$ is finite. The recent result of Aryan \cite[Corollary 1.1]{Ary} states that, for square-free $M$,
\begin{equation}\label{1.8}
S(R_M)\ll M2^{\o(M)}\log M \prod_{p|M}\left(1+\frac{1}{\sqrt{p}}\right)\left(1-\frac1p\right).
\end{equation}

A classical way to estimate $S(R_M)$ is to study
$$
\mathcal{X}_2(R_M; h):=\sum_{n=1}^M\left(\sum_{\substack{m=1 \\ n+m\in R_M}}^h 1 - h\frac{|R_M|}{M} \right)^2
$$
(which is the analog of (\ref{1.5}) for $R_M$). In \cite{Ary}, some lower and upper bounds for $\mathcal{X}_2(R_M; h)$ are obtained in the case of square-free $M$; the upper bound implies (\ref{1.8}) similarly to the proof of Corollary 1 from \cite{MV}. This approach may be useful to answer the aforementioned question of whether $\s(R_M)$ is bounded, but it does not yield to an asymptotic formula for $S(R_M)$ and thus we use another technique.

Let $K_l(M)$ be the number of gaps of length $l$ in $R_M$, that is, the number of elements $x\in R_M$ such that $x+l\in R_M$ but $x+j\notin R_M$ for $1\leq j\leq l-1$. To deal with $S(R_M)$, we write it as $\sum_{l\geq1}K_l(M)l^2$ and, choosing a parameter $D$, divide the range of summation into small ($l\leq D$) and large ($l>D$) gaps. For computing the contribution of small gaps in the setting of Theorem \ref{th1.1} one may use the limit distribution (\ref{1.6}) of the distances between quadratic residues. Nevertheless, in order to be able to compute the second term in the asymptotics for $S(R_M)$, one should avoid $o(1)$-terms which occur in (\ref{1.6}). Instead, we apply estimates of complete characters sums and obtain an asymptotic formula of $K_l(M)$ for small $l$. To bound $K_l(M)$ for large $l$, we use Burgess' estimates for character sums over short intervals.


It is worth mentioning that there are some further questions of interest in this area. For example, can we obtain good estimates for the average $\frac1X\sum_{M\leq X}S(R_M)$? What can be said about the density of the set $\{M\in\N: |\s(R_M)-1|>\d\}$?
We hope to make progress on these questions in a future paper.

In Sections \ref{Sec2}-\ref{Sec5} we will be simultaneously proving Theorems \ref{th1.1} and \ref{th1.4}. Section \ref{Sec2} is devoted to general results concerning additive configurations related to quadratic residues. In Section \ref{Sec3} we find the asymptotics for the contribution of small gaps into $S(R_M)$, and in Section \ref{Sec4} we estimate the contribution of large gaps. In Section \ref{Sec5} we finish the proofs of both Theorems \ref{th1.1} and \ref{th1.4}. Further, in Section \ref{Sec6} we compute the density of the set $\O$. Finally, in Section \ref{Sec7} we prove for completeness the identity (\ref{1.3}).

\section{Additive configurations related to quadratic residues}\label{Sec2} 

To prove our Theorems \ref{th1.1} and \ref{th1.4}, we will need estimates for the number of some additive configurations related to quadratic residues. Let $M$ be a positive integer; for disjoint subsets $C_1$ and $C_2$ of $\Z/M\Z$ we define the sets
$$
R_M(C_1)=\{x\in\Z/M\Z: x+c_1\in R_M \mbox{ for all $c_1\in C_1$} \},
$$
$$
N_M(C_2)=\{x\in\Z/M\Z: x+c_2\notin R_M   \mbox{ for all $c_2\in C_2$} \},
$$
and
$$
RN_M(C_1,C_2)=R_M(C_1)\cap N_M(C_2)
$$
(we do not exclude the cases where $C_1$ or $C_2$ is empty). The lemmas from this section provide bounds for these quantities in the cases of prime and square-free moduli. For $x\in\Z/M\Z$, we denote by $R_M(x)$ the indicator $1_{x\in R_M}$; the letter $\theta$ is used for a number bounded in magnitude by $1$, which can vary from line to line.

We will rely on the following classical estimate for complete character sums.

\begin{lem}\label{lem2.1} 
Let $p$ be a prime, $\chi_1,\ldots,\chi_r$ be characters modulo $p$ with some $\chi_i\neq\chi_0$, where $\chi_0$ is the principal character, and $a_1,\ldots,a_r$ be distinct elements of $\Z/p\Z$. Then
$$
\left|\sum_{x\in\Z/p\Z}\chi_1(x+a_1)\chi_2(x+a_2)\ldots\chi_r(x+a_r)\right|\leq rp^{1/2}.	
$$	
\end{lem}

\begin{proof}
See \cite[Lemma 1]{Joh}.
\end{proof}

By a standard argument, one can obtain from here the asymptotic formula for $|RN_p(C_1,C_2)|$ provided that the sets $C_1$ and $C_2$ are small.

\begin{lem}\label{lem2.2}
Let $p$ be a prime. Then for any disjoint subsets $C_1$ and $C_2$ of $\Z/p\Z$,
\begin{equation*} 
|RN_p(C_1,C_2)|= p2^{-(|C_1|+|C_2|)}+0.5\theta(|C_1|+|C_2)|(p^{1/2}+1).
\end{equation*}
\end{lem}

\begin{proof} We use the method from \cite{Dav}. Let $\chi$ be the Legendre symbol modulo $p$ and $1_p(x)$ be the indicator $1_{x=0}$. Then $R_p(x)=\frac12(1+\chi(x)+1_p(x))$
and $1-R_p(x)=\frac12(1-\chi(x)-1_p(x))$. Since the functions $R_p(x)$, $1-R_p(x)$, and $\frac{1\pm\chi(x+c_i)}{2}$ are bounded in magnitude by $1$, we have
\begin{multline*}
|RN_p(C_1,C_2)|=\sum_{x\in\Z/p\Z}\prod_{c_1\in C_1}R_p(x+c_1)\prod_{c_2\in C_2}(1-R_p(x+c_2))=\\
\sum_{x\in\Z/p\Z}\prod_{c_1\in C_1}\frac{1+\chi(x+c_1)+1_p(x+c_1)}{2}\prod_{c_2\in C_2}\frac{1-\chi(x+c_2)-1_p(x+c_2)}{2}=\\
2^{-(|C_1|+|C_2|)}\sum_{x\in\Z/p\Z}\prod_{c_1\in C_1}(1+\chi(x+c_1))\prod_{c_2\in C_2}(1-\chi(x+c_2))+0.5\theta(|C_1|+|C_2|).
\end{multline*}
Denote $k=|C_1|+|C_2|$. Expanding the product in the last expression, we obtain the term $p$ and $2^k-1$ character sums; there are exactly ${k\choose r}$ of them of the type $\sum_{x\in\Z/p\Z}\prod_{j=1}^r\chi(x+a_j)$, each of which is bounded by $rp^{1/2}$ by Lemma \ref{lem2.1}. Since $\sum_{r=1}^k{k\choose r}r=k2^{k-1}$, we get
\begin{equation*}
2^k|RN_p(C_1,C_2)|=p+0.5\theta k2^k(p^{1/2}+1),
\end{equation*} 
and the claim follows.   
\end{proof}

Now we extend Lemma \ref{lem2.2} to the case of square-free moduli.  

\begin{lem}\label{lem2.3} 
Let $m$ be a square-free positive integer, $p_1=P^-(m)$ and $t=\o(m)$, and let $C_1$ and $C_2$ be disjoint subsets of nonnegative integers less than $p_1$, $|C_1|=s_1$, $|C_2|=s_2$. Then for all $r\in\N$ 
\begin{equation*} 
|RN_m(C_1,C_2)|=m2^{-s_1t}(1-2^{-t})^{s_2}+O(m2^{-s_1t}\left(E_1+E_2\right)),
\end{equation*}
where 
$$
E_1= \sum_{k=0}^{\min\{r,s_2\}}{s_2\choose k}2^{-kt}E(s_1+k),
$$
$$
E(\nu)=\sum_{d|m,d>1}\frac{\left(\nu2^{\nu}\right)^{\o(d)}}{d^{1/2}}, \quad 	E_2=\sum_{k=r}^{s_2}{s_2\choose k}2^{-kt}
$$
(we adopt the convention that $E_2=0$ for $r>s_2$).
\end{lem}

In practice, we will apply this lemma to situations where $m$ is the product of small amount of ``large'' primes; for such $m$ one can control the error terms. 

Note that the probability that a random element $x\in \Z/m\Z$ is a quadratic residue is about $2^{-t}$. The main term $m2^{-s_1t}(1-2^{-t})^{s_2}$ in this lemma agrees with this, also telling us that the events $x+c_1\in R_m$, $c_1\in C_1$, and $x+c_2\notin R_m$, $c_2\in C_2$, are nearly independent, provided that the error terms are small. 

\begin{proof} We may assume $r\leq s_2+1$. For a subset $C\subseteq C_2$, we write for brevity
$$
R_C:=R_m(C_1\cup C)=\{x\in\Z/m\Z: x+c\in R_m \mbox{ for all $c\in C_1\cup C$}\}.
$$
Note that $|RN_m(C_1,C_2)|=\left|R_{\varnothing}\setminus\bigcup_{c\in C_2}R_{\{c\}}\right|=|R_{\varnothing}|-|\bigcup_{c\in C_2}R_{\{c \}}|$. Hence by the inclusion-exclusion principle
\begin{equation}\label{2.2}
|RN_m(C_1,C_2)|=|R_{\varnothing}|-\sum_{|C|=1}|R_{C}|+\sum_{|C|=2}|R_C|-\ldots+(-1)^{s_2}|R_{C_2}|,
\end{equation}
and by the Bonferroni inequalities (see, for example, \cite[Ch.1, Exercise 1.1.3]{TV}) we have
\begin{equation}\label{2.3} 
|RN_m(C_1,C_2)|=|R_{\varnothing}|-\sum_{|C|=1}|R_C|+\sum_{|C|=2}|R_C|-\ldots+(-1)^{r-1}\sum_{|C|=r-1}|R_C|+O\left(\sum_{|C|=r}|R_C|\right)
\end{equation}
for all $r\leq s_2$.

Now we fix $C\subseteq C_2$ and work with $|R_C|$. By multiplicity and the assumption that elements of $C_1\bigsqcup C_2$ are less than $p_1$ (and, hence, are distinct modulo $p$ for any $p|m$), we can rewrite it as
\begin{equation*}
|R_C|=\prod_{p|m}|R_p(C_1\cup C)|.
\end{equation*}
Let $|C|=k$. By Lemma \ref{lem2.2}, $|R_p(C_1\cup C)|=p2^{-(s_1+k)}\left(1+\theta(s_1+k)2^{s_1+k}p^{-1/2}\right)$ for some $|\theta|\leq1$. Thus, we see that 
\begin{equation}\label{2.4} 
|R_C|=m2^{-(s_1+k)t}+O\left(m2^{-(s_1+k)t}E(s_1+k)\right),
\end{equation}
with $E(s_1+k)$ defined in the lemma.

We first consider the case $r=s_2+1$. By (\ref{2.2}) and (\ref{2.4}),
\begin{multline*}
2^{s_1t}m^{-1}|RN_m(C_1,C_2)|=1-{s_2\choose 1}2^{-t}+{s_2\choose 2}2^{-2t}+\ldots  +(-1)^{s_2}2^{-s_2t}\\+O\left(\sum_{k=0}^{s_2}{s_2\choose k}2^{-kt}E(s_1+k)\right)=
(1-2^{-t})^{s_2}+O(E_1), 
\end{multline*}
and the claim follows.

Now suppose that $r\leq s_2$. From (\ref{2.3}) and (\ref{2.4}) we obtain 
\begin{multline*}
2^{s_1t}m^{-1}RN_m(C_1,C_2)=1-{s_2\choose 1}2^{-t}+{s_2\choose 2}2^{-2t}+\ldots +(-1)^{r-1}{s_2\choose {r-1}}2^{-(r-1)t}\\
+O\left(\sum_{k=0}^{r}{s_2\choose k}2^{-kt}E(s_1+k)+{s_2\choose {r}}2^{-rt}\right)=
(1-2^{-t})^{s_2}+E, 
\end{multline*}
where
\[
E\ll E_1+{s_2\choose {r}}2^{-rt}+\left|\sum_{k=r}^{s_2}(-1)^k{s_2 \choose k}2^{-kt}\right| \ll E_1+\sum_{k=r}^{s_2}{s_2 \choose k}2^{-kt}=E_1+E_2,
\]
as desired. This concludes the proof.    
\end{proof}

The main results of this section are the following two corollaries.

\begin{cor}\label{cor2.4}
Let $m$ be a square-free positive integer with $p_1=P^-(m)$ large enough, $t=\o(m)$, and let $C_1=\{0,l\}$ and $C_2$ be disjoint subsets of nonnegative integers less than $p_1$. Suppose that $s=|C_2|\leq 0.12\cdot2^t\log p_1$ and $t\leq p_1^{0.01}$. Then
$$ |RN_m(C_1,C_2)|=m2^{-2t}(1-2^{-t})^s+O(m2^{-2t}(1+2^{1-t})^sp_1^{-0.48}).	
$$	
\end{cor}

\begin{proof} We set $r=\left\lfloor\frac{0.48\log p_1}{\log2}\right\rfloor$ in the previous lemma and estimate the quantities $E(k+2)$ for $0\leq k\leq r$.
Since then $t(k+2)2^{k+2}<p_1^{1/2}$, we have
\begin{multline*}
E(k+2)=\sum_{d|m,d>1}\frac{((k+2)2^{k+2})^{\o(d)}}{d^{1/2}} =\prod_{p|m}\left(\frac{(k+2)2^{k+2}}{p^{1/2}}+1\right)-1  \\ \leq \exp\left((k+2)2^{k+2}\sum_{p|m}p^{-1/2}\right)-1 \leq  
\exp(t(k+2)2^{k+2}p_1^{-1/2})-1 \ll tr2^kp_1^{-1/2}.
\end{multline*}
We thus see that
$$
E_1=\sum_{k=0}^{\min\{r,s\}}{s \choose k}2^{-kt}E(k+2) \ll trp_1^{-1/2}\sum_{k=0}^s{s\choose k}2^{(1-t)k} = (1+2^{1-t})^strp_1^{-1/2} \ll (1+2^{1-t})^sp_1^{-0.48}.
$$
Further, 
$$
{s \choose k}2^{-kt}\leq \frac{s^k}{k!2^{kt}} \ll   \left(\frac{se}{k2^t}\right)^k.
$$
Besides, we have $se\leq 0.12e2^t\log p_1 < 0.5r2^t$, and therefore
$$
E_2=\sum_{k=r}^s{s \choose k}2^{-kt}  \ll \sum_{k=r}^{\infty}\left(\frac{se}{r2^t}\right)^k \ll 2^{-r} \ll  p_1^{-0.48}.
$$
The claim now follows from Lemma \ref{lem2.3}.  
\end{proof}

\begin{cor}\label{cor2.5}
Let $m$ be a square-free positive integer with $p_1=P^-(m)$ large enough, $t=\o(m)$, and let $C$ be a subset of nonnegative integers less than $p_1$. Suppose that $|C|\geq \a2^t\log p_1$ for some fixed $\a\in(0,0.12)$ and $t\leq p_1^{0.01}$. Then
$$
|N_m(C)| \ll mp_1^{-\a}. 
$$
\end{cor}

\begin{proof} Let $s=|C|$; since $N_m(C')\leq N_m(C'')$ whenever $C'\supseteq C''$, we may assume that $s=\lfloor\a2^t\log p_1\rfloor$. Arguing as in the proof of the previous lemma and using the inequality $(1+x)\leq e^x$, $x\in\R$, we get from Lemma \ref{lem2.3}
$$
m^{-1}|N_m(C)| \ll (1-2^{-t})^s+(1+2^{1-t})^sp_1^{-0.48}\ll p_1^{-\a}+p_1^{2\a-0.48} \ll p_1^{-\a},
$$
as desired.   
\end{proof}

\section{Contribution of small gaps}\label{Sec3}  

To work with $S(R_M)$, we define $K_l(M)$ to be the number of $x\in\Z/M\Z$ with $x,x+l\in R_M$ and $x+j\notin R_M$ for all $1\leq j\leq l-1$, and write
\begin{equation}\label{3.0}
S(R_M)=\sum_{l\leq D}K_l(M)l^2 + \sum_{l>D}K_l(M)l^2
\end{equation}
for an appropriate $D=D(M)$. The aim of this section is to find the asymptotic formula for $\sum_{l\leq D}K_l(M)l^2$ with $M$ from Theorems \ref{th1.1} and \ref{th1.4}.

\smallskip 

For $2\leq j\leq A$ and $i\in \Z/A\Z$, we set $I(i,j)=\{i+1,...,i+j-1\}\pmod{A}$ and
\begin{equation}\label{3.1} 
r(i,j)=\#\big(I(i,j)\cap R_A\big);
\end{equation}
we also adopt the convention that $r(i,1)=0$. Obviously, $0\leq r(i,j)\leq |R_A|$.

The next lemma is a general statement which allows us to convert estimates for additive configurations in quadratic residues to an expression for $\sum_{l\leq D}K_l(M)l^2$.

\begin{lem}\label{lem3.1}
Let $1\leq D<m$, $(A,m)=1$, $M=Am$, $t=\o(m)$, and $D_0=D|R_A|A^{-1}+|R_A|$. Suppose that for any $s\leq D_0$ we have an estimate of the type
\begin{equation}\label{3.2} 
\sum_{u\in\Z/m\Z} R_m(u)R_m(u+l)\prod_{j=1}^{s}(1-R_m(u+\nu_j))=m2^{-2t}(1-2^{-t})^s+O(L(m,s))
\end{equation}
with some nonnegative real numbers $L(m,s)$, uniformly in distinct positive integers $\nu_1,\ldots,\nu_s,l$ not exceeding $D$. Then 
\begin{equation}\label{3.3}
\sum_{l\leq D} K_l(M)l^2 =
m2^{-2t}\sum_{j=1}^{A}\sum_{i\in R_A\cap(R_A-j)}\sum_{Ak+j\leq D}y_t^{|R_A|k+r(i,j)}(Ak+j)^2+O(E(M,D_0)),
\end{equation}
where $y_t=1-2^{-t}$ and
\begin{equation}\label{3.4}
E(M,D_0)=A^3|R_A|\sum_{s\leq D_0}L(m,s)(s+1)^2.	
\end{equation}
\end{lem}

\begin{proof} Since $(A,m)=1$, we have $R_M(x)=R_A(x)R_m(x)$, and hence
\begin{equation}\label{3.5}
R_M(x)=\begin{cases}
R_m(x),  & x\in R_A,\\
0,  & \mbox{otherwise.}
\end{cases}
\end{equation}
Fix $l\leq D$. We can write 
\begin{equation*}
K_l(M)=\sum_{x\in\Z/M\Z}R_M(x)R_M(x+l)\prod_{\nu=1}^{l-1}(1-R_M(x+\nu))=\sum_{i=0}^{A-1}K_{l,i}(M),
\end{equation*}
where
\begin{equation*}
K_{l,i}(M)=\sum_{x\in\Z/M\Z: x\equiv i\pmod{A}}R_M(x)R_M(x+l)\prod_{\nu=1}^{l-1}(1-R_M(x+\nu)).
\end{equation*}
Further, using (\ref{3.5}) we get
\begin{equation}\label{3.6}
K_l(M)=\sum_{i\in R_A\cap(R_A-l)}K_{l,i}(M).
\end{equation}
Now fix  $i\in R_A\cap(R_A-l)$. Let $\nu_1,\ldots,\nu_{s(l,i,A)}$ be the numbers $\nu\in\{1,\ldots,l-1\}$ such that $(\nu+i)\pmod{A}\in R_A$. We write $l=Ak+j$ for some $k\geq0$ and $1\leq j\leq A$; then $l/A-1\leq k\leq l/A$ and
\begin{equation}\label{3.7}
s(l,i,A)=|R_A|k+r(i,j) = l|R_A|A^{-1}+\theta'|R_A|
\end{equation}
for some $\theta'\in[-1,1]$. In particular, $0\leq s(l,i,A)\leq D|R_A|A^{-1}+|R_A|=D_0$ for any $l\leq D$. Taking into account (\ref{3.5}) and using the fact that the map $u\mapsto Au+i$ is a bijection from $\Z/m\Z$ to $\Z/m\Z$, we obtain from (\ref{3.2}) 
\begin{multline*}
K_{l,i}(M)=\sum_{x\in\Z/M\Z: \, x\equiv i\pmod{A}}R_m(x)R_m(x+l)\prod_{j=1}^{s(l,i,A)}(1-R_m(x+\nu_j))\\
=\sum_{u\in\Z/m\Z} R_m(Au+i)R_m(Au+i+l)\prod_{j=1}^{s(l,i,A)}(1-R_m(Au+i+\nu_j))\\
=\sum_{u\in\Z/m\Z} R_m(u)R_m(u+l)\prod_{j=1}^{s(l,i,A)}(1-R_m(u+\nu_j))=m2^{-2t}y_t^{s(l,i,A)}+O(L(m,s(l,i,A))).
\end{multline*}
Since $l\equiv j\pmod{A}$, from here, (\ref{3.6}), and (\ref{3.7}) we have
\begin{equation*}\label{N_l}
K_l(M)=m2^{-2t}\sum_{i\in R_A\cap(R_A-j)}y_t^{|R_A|k+r(i,j)}+O\left(\sum_{i\in R_A\cap (R_A-l)}L(m,s(l,i,A))\right).
\end{equation*}
Summing this estimate over $l\leq D$, we get 
\begin{equation*}
\sum_{l\leq D} K_l(M)l^2 =
m2^{-2t}\sum_{j=1}^{A}\sum_{i\in R_A\cap(R_A-j)}\sum_{Ak+j\leq D}y_t^{|R_A|k+r(i,j)}(Ak+j)^2+O(\widetilde{E}(M,D_0)),
\end{equation*}
where
$$
\widetilde{E}(M,D_0)=\sum_{l\leq D}\sum_{i\in R_A\cap (R_A-l)}L(m,s(l,i,A))l^2.
$$
Note that (\ref{3.7}) implies $l\leq s(l,i,A)A|R_A|^{-1}+A$ for any $i$. Further, if $l$ runs from $Ak+1$ to $Ak+A$ and $i$ runs over $R_A\cap (R_A-l)$, then $s(l,i,A)$ varies between $|R_A|k$ and $|R_A|k+|R_A|$. Hence, for any $s\geq0$ there are at most $O(A|R_A|)$ pairs $(l,i)$ with $s(A,l,i)=s$. Thus,
$$
\widetilde{E}(M,D_0)\ll \sum_{s=0}^{D_0}L(m,s)A|R_A|\left(sA|R_A|^{-1}+A\right)^2 \ll A^3|R_A|\sum_{s\leq D_0}L(m,s)(s+1)^2.
$$
This concludes the proof. 
\end{proof}

Our goal now is to apply Lemma \ref{lem3.1} together with the results from Section \ref{Sec2} to rewrite $\sum_{l\leq D}K_l(M)l^2$ in terms of the function 
\begin{equation}\label{3.8}
g_A(y)=\sum_{j=1}^{A}\sum_{i\in R_A\cap(R_A-j)}y^{r(i,j)}\sum_{k\geq 0}y^{|R_A|k}(Ak+j)^2, \quad y\in(0,1)
\end{equation}
(this expression will be simplified in Section \ref{Sec5}). Note that $g_A$ does not depend on $m$. 

We also introduce the function
\begin{equation}\label{3.9}
G_A(y,D)=\sum_{j=1}^{A}\sum_{i\in R_A\cap(R_A-j)}y^{r(i,j)}\sum_{Ak+j>D}y^{|R_A|k}(Ak+j)^2,
\end{equation}	
which is the ``tail'' of $g_A(y)$. We will need the following bound.

\begin{lem}\label{lem3.2}
Let $D\geq 2A$. Then
$$
G_A(y,D)\ll  A^4y^{|R_A|(\lfloor D/A\rfloor -2)}\left(1-y^{|R_A|}\right)^{-1}\left(D^2A^{-2}+\left(1-y^{|R_A|}\right)^{-2}\right).
$$
\end{lem}
 
\begin{proof}
Firstly, we have
\begin{equation*}
\sum_{j=1}^A\sum_{i\in R_A\cap(R_A-j)}1\leq \sum_{i=1}^AR_A(i)\sum_{j=1}^AR_A(i+j)=|R_A|^2.
\end{equation*}
Hence,
\begin{multline*} 
G_A(y,D)\leq \sum_{j=1}^A\sum_{i\in R_A\cap(R_A-j)}\sum_{Ak+A>D}y^{|R_A|k}(Ak+A)^2 \\ 
\leq A^2|R_A|^2\sum_{k+1>D/A}y^{|R_A|k}(k+1)^2 \ll A^2|R_A|^2y^{-|R_A|}\sum_{k>D/A-1}k^2y^{|R_A|k}.
\end{multline*}
Now we note that for all positive integer $N$ and $x\in (0,1)$, 
\begin{equation*}
\sum_{n\geq N}n^2x^n \ll x^N(1-x)^{-1}(N^2+(1-x)^{-2})
\end{equation*} 
(this can be obtained by differentiating twice the equality $\sum_{n\geq N}x^n=x^N(1-x)^{-1}$), and therefore
$$
G_A(y,D)\ll A^4y^{|R_A|(\lfloor D/A\rfloor -2)}\left(1-y^{|R_A|}\right)^{-1}\left(D^2A^{-2}+\left(1-y^{|R_A|}\right)^{-2}\right),
$$
as desired.	
\end{proof}	 

Now we are ready to present the main results of this section. Let $\e_0\in(0,1/2)$ be a fixed number to be chosen later.

\begin{cor}\label{cor3.3}
Let $M=Ap$ with prime $p$ and $A\leq 0.1p^{\e_0}$. Then	for $D=A|R_A|^{-1}p^{\e_0}$ we have
$$
\sum_{l\leq D}K_l(M)l^2 = 0.25g_A(0.5)p+O(A^4p^{1/2+4\e_0}).
$$
\end{cor}

\begin{proof} We have $10A\leq D<p$ and $D_0=D|R_A|A^{-1}+|R_A|=p^{\e_0}+|R_A|\leq 2p^{\e_0}$. Note also that $A\leq 0.1p^{\e_0}$ yields $(A,p)=1$. Lemma \ref{lem2.2} then gives us (\ref{3.2}) with $t=1$ and $L(p,s)=(s+2)p^{1/2}$. 
By Lemma \ref{lem3.1} and the definition (\ref{3.9}) of $G_A(y,D)$,
$$
\sum_{l\leq D}K_l(M)l^2 = 0.25g_A(0.5)p+O\Big(E(M,D_0)+G_A(0.5,D)p\Big),
$$
and it suffices to bound the error terms. By the definition (\ref{3.3}) of $E(M_0,D)$, we see that
$$
E(M,D_0)\ll A^3|R_A|p^{1/2}\sum_{s\leq D_0}s^3 \ll A^4p^{1/2+4\e_0}.
$$	
Further, since $(1-2^{-|R_A|})^{-1}\leq 2$ and $\lfloor D/A\rfloor-2\geq 0.5|R_A|^{-1}p^{\e_0}$, Lemma \ref{lem3.2} implies 
$$
G_A(0.5,D)p \ll A^4p^{1+2\e_0}2^{-0.5p^{\e_0}} \ll 1.
$$
The claim follows. 
\end{proof}

\begin{cor}\label{cor3.4}
Let $M=Am\in \O$ and $D=0.11A|R_A|^{-1}2^t\log p_1$, where $t=\o(m)$ and $p_1=P^-(m)$. Then
$$
\sum_{l\leq D}K_l(M)l^2=m2^{-2t}g_A(y_t)+O(MA^32^tp_1^{-0.1}).
$$ 
\end{cor}

\begin{proof} Note that $M\in\O$ implies $A\leq 2^t$, $2A\leq D<p_1<m$, and $t\leq p_1^{0.01}$. We also have $D_0=0.11\cdot2^t\log p_1+|R_A|\leq 0.12\cdot2^t\log p_1$. Now Corollary \ref{cor2.4} gives us (\ref{3.2}) with
$$
L(m,s) = m2^{-2t}(1+2^{1-t})^sp_1^{-0.48};
$$
Hence Lemma \ref{lem3.1} yields
$$
\sum_{l\leq D}K_l(M)l^2=m2^{-2t}g_A(y_t)+O\Big(E(M,D_0)+m2^{-2t}G_A(y_t,D)\Big),
$$
where
\begin{multline*}
E(M,D_0) \ll A^3|R_A|m2^{-2t}p_1^{-0.48}\sum_{s\leq D_0}(1+2^{1-t})^s(s+1)^2 \\
\ll MA^32^{-t}p_1^{-0.48}D_0^2(1+2^{1-t})^{D_0} \ll MA^32^tp_1^{-0.47}(1+2^{1-t})^{D_0} \leq MA^32^tp_1^{-0.23}
\end{multline*}
(here and in what follows we use the inequality $1+x\leq e^x$, $x\in\R$). Further, $y_t^{|R_A|}\leq y_t = 1-2^{-t}$, and hence $\big(1-y_t^{|R_A|}\big)^{-1}\leq 2^t$; also $1/2\leq y_t<1$ and $y_t^{-|R_A|}\ll 1$. Using Lemma \ref{lem3.2}, we see that
$$
G_A(y_t,D)\ll A^4y_t^{|R_A|D/A}2^t\left( D^2A^{-2}+2^{2t}\right)
\ll A^42^{3t}((\log p_1)^2+1)p_1^{-0.11} \ll A^42^{3t}p_1^{-0.1},
$$ 
and so $$
m2^{-2t}G_A(y_t,D)\ll MA^32^tp_1^{-0.1}.
$$
This concludes the proof. 
\end{proof}

\section{Contribution of large gaps}\label{Sec4}
In view of the results of the previous section, to rewrite $S(R_M)$ in terms of the function $g_A$, it remains to bound the sums $\sum_{l>D}K_l(M)l^2$. To work with large gaps, we will need famous Burgess' estimates for character sums.

\begin{lem}\label{lem4.1} 
Let $k$ be square-free, $\chi$ be a nonprincipal character modulo $k$, and $\e>0$, $r\in\N$. Then for any $x\in\Z/k\Z$ and $h\in\N$ 
$$
\left|\sum_{i=0}^{h-1}\chi(x+i)\right|\ll_{r,\e}h^{1-1/r}k^{(r+1)/(4r^2)+\e}. 
$$	
\end{lem}

\begin{proof} See \cite[Theorem 2]{burg1}.    
\end{proof}

\begin{lem}\label{lem4.2} 
Let $k$ be square-free, $\chi$ be a primitive character modulo $k$, and $\e>0$, $r\in\N$. Then 
$$
\sum_{x\in\Z/k\Z}\left|\sum_{i=0}^{h-1}\chi(x+i)\right|^{2r} \ll_{r,\e} kh^r+k^{1/2+\e}h^{2r}.
$$	
\end{lem}

\begin{proof} See \cite[Lemma 8]{burg2}.  
\end{proof}

Let $h\in\N$. For brevity, we denote by $N_M(h)=N_M(\{0,1,...,h-1\})$ the set of $x\in\Z/M\Z$ such that all elements $x,x+1,\ldots,x+h-1$ are quadratic nonresidues modulo $M$. The heart of this section is the following lemma.

\begin{lem} \label{lem4.3}
Let $M$ be a square-free positive integer. Suppose that $h\geq (2C\tau(M))^{50}$,
where $C$ is a large enough absolute constant. Then 
$$
|N_M(h)| \ll M\tau(M)^4h^{-1.47}. 
$$
\end{lem} 

\begin{proof} Let $\L$ be the set of nontrivial real characters modulo $M$, and $\chi_p$ be the Legendre symbol modulo $p$. For any quadratic nonresidue $x\in\Z/M\Z$ we have 
$$
\sum_{\chi\in\L}\chi(x)=\prod_{p|M}(1+\chi_p(x))-1=-1;
$$
therefore, for any $x\in N_M(h)$ 
\begin{equation}\label{4.1} 
\left|\sum_{\chi\in\L}\sum_{i=0}^{h-1}\chi(x+i)\right|=h.
\end{equation}
Let $k(\chi)$ denote the conductor of $\chi\in\L$. We define
$$
\L_1=\{\chi\in\L: k(\chi)\leq h^3\}, \quad \L_2=\{\chi\in L: k(\chi)>h^3\}.
$$
First, we want to show that
\begin{equation}\label{4.2} 
\sum_{\chi\in\L_1}\left|\sum_{i=0}^{h-1}\chi(x+i)\right|\leq h/2.
\end{equation}
For any $\chi\in\L_1$ Lemma \ref{lem4.1} implies
$$
\left|\sum_{i=0}^{h-1}\chi(x+i)\right|\ll_{r,\e} h^{1-1/r}h^{3(r+1)/(4r^2)+3\e}.
$$
Taking $r=6$ and $\e=10^{-5}$, we obtain
$$
\sum_{\chi\in\L_1}\left|\sum_{i=0}^{h-1}\chi(x+i)\right|\leq Ch^{1-1/48+0.0003}|\L_1|\leq Ch^{0.98}|\L_1|
$$
for some absolute constant $C>0$, and (\ref{4.2}) follows from $C|\L_1|\leq C|\L|=C\tau(M)\leq 0.5h^{0.02}$ (the last inequality here is equivalent to the assumption of the lemma).

From (\ref{4.1}) and (\ref{4.2}) we see that
$$
\sum_{\chi\in\L_2}\left|\sum_{i=0}^{h-1}\chi(x+i)\right|\geq h/2,
$$
for all $x\in N_M(h)$, and thus
\begin{equation}\label{4.3} 
\sum_{x\in N_M(h)}\sum_{\chi\in\L_2}\left|\sum_{i=0}^{h-1}\chi(x+i)\right|\geq |N_M(h)|h/2.
\end{equation}
On the other hand, for any $\chi\in\L_2$ we have from Lemma \ref{lem4.2} (with $\e=0.01$)
\begin{multline*}
\sum_{x\in N_M(h)}\left|\sum_{i=0}^{h-1}\chi(x+i)\right|^{2r}\leq \sum_{x\in\Z/M\Z}\left|\sum_{i=0}^{h-1}\chi(x+i)\right|^{2r}\\ \ll \frac{M}{k_{\chi}} \sum_{x\in \Z/k_{\chi}\Z}\left|\sum_{i=0}^{h-1}\chi(x+i)\right|^{2r}\ \ll_r Mh^r+Mh^{2r}k_{\chi}^{-0.49}.
\end{multline*}
Setting now $r=2$, for all $\chi\in\L_2$ we have
$$
\sum_{x\in N_M(h)}\left|\sum_{i=0}^{h-1}\chi(x+i)\right|^{4} \ll Mh^{2.53}.
$$
Using H\"older's inequality, we find from (\ref{4.3})
$$
|N_M(h)|h \ll \sum_{\chi\in\L_2}\sum_{x\in N_M(h)}\left|\sum_{i=0}^{h-1}\chi(x+i)\right| \ll \tau(M)|N_M(h)|^{3/4}(Mh^{2.53})^{1/4},
$$
and, hence,
$$
|N_M(h)|\ll M\tau(M)^4h^{-1.47},  
$$
as desired.
\end{proof}

\begin{rem}\label{rem4.4}
One can deduce from this argument that the size of the maximal gap in the set $R_M$ is at most $O(M^{1/3})$, but we do not need this estimate here.	
\end{rem}

Let $R_M=\{0=r_1<r_2<\ldots<r_{|R_M|}\}$ be the set of quadratic residues modulo $M$. 

\begin{lem}\label{lem4.5}
We have
$$
\sum_{2^{\nu+1}\leq r_i-r_{i-1}<2^{\nu+2}}
(r_i-r_{i-1})^2 \leq |N_M(2^\nu)|2^{\nu+4}.
$$
\end{lem}

\begin{proof} If $r_i-r_{i-1}\geq 2^{\nu+1}$, then for each $1\leq k\leq 2^{\nu}$, the element $r_{i-1}+k$ belongs to $N_M(2^{\nu})$. Therefore the number of such $i$ is at most $2^{-\nu}|N_M(2^{\nu})|$. The claim follows. 
\end{proof}

Now we are in position to combine the estimates for $|N_M(h)|$ with the results from the previous section to get the desired expression for $S(R_M)$. We start with the situation of Theorem \ref{th1.4}.

\begin{cor}\label{cor4.6}
Let $M=Ap$, where $A$ is square-free, $A\leq p^{0.1}$, and $p$ is large enough. Then	
$$
S(R_M) = 0.25g_A(0.5)p+O\left(A^4p^{0.95}\right).
$$
\end{cor}

\begin{proof} Let $\e_0=\frac{1}{2\cdot4.47}=0.11...$ and $D=A|R_A|^{-1}p^{\e_0}$. Then $A\leq 0.1p^{\e_0}$, and the equality (\ref{3.0}) together with Corollary \ref{cor3.3} gives us
\begin{equation}\label{4.4}
S(R_M)=0.25g_A(0.5)p+O\left(A^4p^{1/2+4\e_0}+\sum_{l>D}K_l(M)l^2\right). 
\end{equation}	
Using Lemma \ref{lem4.5}, we obtain
$$
\sum_{l>D}K_l(M)l^2 \leq \sum_{\nu\,:\, 2^{\nu}>D}\sum_{2^{\nu-1}\leq r_j-r_{j-1}<2^{\nu}}(r_j-r_{j-1})^2 \ll \sum_{\nu\,:\, 2^{\nu}>D}|N_M(2^{\nu-2})|2^{\nu}. 
$$
Let $C$ be the constant from Lemma \ref{4.3}. Since $\tau(n)\ll_{\e}n^{\e}$ for any $\e>0$ and $p$ is large enough, we have $(4C\tau(A))^{50} \leq 0.25p^{\e_0} \leq 0.25D$. Thus $2^{\nu}>D$ implies $2^{\nu-2}>(2C\tau(M))^{50}$. Now Lemma \ref{lem4.3} gives us
$$
\sum_{l>D}K_l(M)l^2\ll M\tau(M)^4\sum_{\nu\,:\, 2^{\nu+2}>D}2^{-0.47\nu} \ll A\tau(A)^4pD^{-0.47} \ll A^2p^{1-0.47\e_0}.
$$
Combining this with (\ref{4.4}) and taking into account that $1-0.47\e_0=1/2+4\e_0<0.95$, we get the desired result.
\end{proof}

Now we turn to the situation of Theorem \ref{th1.1} and work with $M\in\O$. 

\begin{cor}\label{cor4.7}
Let $M=Am\in \O$ and $t=\o(m)$. Then	
$$
S(R_M) = m2^{-2t}g_A(y_t) + O(M(\log M)^{-1}).
$$	
\end{cor}

\begin{proof}
We set $D=0.11A|R_A|^{-1}2^t\log p_1$ (where, as usual, $p_1=P^-(m)$)	
and $D'=2^{51\o(M)}$; note that $D<D'<0.5p_1$ provided that the constant $C_0$ from the definition of $\O$ is large enough. Corollary \ref{cor3.4} allows us to write
\begin{multline}\label{4.5}
S(R_M)=\sum_{l\leq D}K_l(M)l^2+\sum_{D<l\leq D'}K_l(M)l^2 + \sum_{l>D'}K_l(M)l^2 \\ =m2^{-2t}g_A(y_t) + O\left(MA^32^tp_1^{-0.1}+\sum_{D<l\leq D'}K_l(M)l^2 + \sum_{l>D'}K_l(M)l^2\right),	
\end{multline}
and it remains to estimate the sums corresponding to ``average'' and ``large'' gaps. 

We begin with the contribution of average gaps. Lemma \ref{lem4.5} gives us	
\begin{equation}\label{4.6}
\sum_{D<l\leq D'} K_l(M)l^2 \leq \sum_{\nu: D<2^{\nu}\leq 2D'}\sum_{2^{\nu-1}\leq r_j-r_{j-1}<2^{\nu}}(r_j-r_{j-1})^2 \ll \sum_{\nu: D<2^{\nu}\leq 2D'}|N_M(2^{\nu-2})|2^{\nu}. 
\end{equation}
Now we fix $\nu$ with $D<2^{\nu}\leq 2D'$. To estimate $|N_M(2^{\nu-2})|$, let us take $x\in N_M(2^{\nu-2})$ and let $x\equiv a\pmod A$. We define
$$
C_a=\left\{0\leq i\leq 2^{\nu-2}-1: a+i \pmod{A} \in R_A \right\}; 
$$
then $x\pmod m \in N_m(C_a)$. Further, all elements of $C_a$ do not exceed $2D'<p_1$,
$$
|C_a|\geq |R_A|\cdot\lfloor 2^{\nu-2}A^{-1}\rfloor \geq |R_A|A^{-1}2^{\nu-2}-|R_A|\geq 0.02\cdot2^t\log p_1
$$
Using Corollary \ref{cor2.5}, we obtain
$$
|N_M(h)|\leq \sum_{a=1}^{A}|N_m(C_a)| \ll Amp_1^{-0.02},
$$
and from (\ref{4.5}) we have
\begin{equation}\label{4.7}
\sum_{D<l\leq D'} K_l(M)l^2 \ll MD'p_1^{-0.02}.	
\end{equation}

Finally, we work with the gaps of length greater than $D'$. As before, Lemma \ref{lem4.5} gives us	
$$
\sum_{l>D'}K_l(M)l^2 \leq \sum_{\nu\,:\, 2^{\nu}>D'}\sum_{2^{\nu-1}\leq r_j-r_{j-1}<2^{\nu}}(r_j-r_{j-1})^2 \ll \sum_{\nu\,:\, 2^{\nu}>D'}|N_M(2^{\nu-2})|2^{\nu}. 
$$
The inequality $2^{\nu}>D'=2^{51\o(M)}$ together with the fact that $\tau(M)=2^{\o(M)}$ is large enough (which is due to the definition of the set $\O$) implies $2^{\nu-2}>(2C\tau(M))^{50}$. Now an appeal to Lemma \ref{lem4.3} shows that
\begin{equation}\label{4.8}
\sum_{l>D'}K_l(M)l^2 \ll M2^{4\o(M)}\sum_{\nu: 2^{\nu}>D'}2^{-0.47\nu} \ll M2^{-19\o(M)}.
\end{equation}

Combining Corollary (\ref{4.5}), (\ref{4.7}), and (\ref{4.8}), we obtain
$$
S(R_M) = m2^{-2t}g_A(y_t) +O\left(MA^32^tp_1^{-0.1}+M2^{51\o(M)}p_1^{-0.02}+M2^{-19\o(M)}\right).	
$$
Recall that $M\in\O$ implies $\o(M)/\log\log M\in[0.9,1.1]$, $A\leq 2^t\leq (\log M)^{1.1\log2}$, and $p_1>(\log M)^{C_0}$. Now by taking $C_0$ large enough we conclude the proof. 
\end{proof}

\section{Completion of the proofs of Theorems \ref{th1.1} and \ref{th1.4}}\label{Sec5}

Having proved Corollaries \ref{cor4.6} and \ref{cor4.7}, we found the assumed asymptotics for $S(R_M)$ in terms of the function $g_A$. Now we are aiming to extract further information about it. Recall that for any $y\in(0,1)$
\begin{equation}\label{5.1}
g_A(y)=\sum_{j=1}^{A}\sum_{i\in R_A\cap(R_A-j)}y^{r(i,j)}\sum_{k\geq 0}y^{|R_A|k}(Ak+j)^2;
\end{equation}
We begin with the following simple lemma, which allows us to simplify this expression.

\begin{lem}\label{lem5.1}
For all $x\in(0,1)$ 
\begin{equation*} 
\sum_{n=0}^{\infty}x^n=\frac{1}{1-x}, \qquad \sum_{n=1}^{\infty}nx^{n}=\frac{x}{(1-x)^2},
\qquad \sum_{n=1}^{\infty}n^2x^{n}=\frac{x(1+x)}{(1-x)^3}.
\end{equation*}
\end{lem} 

\begin{proof} This can be obtained by standard direct calculations.
\end{proof}

We apply this lemma to the inner sum in the definition (\ref{5.1}) of the function $g_A$. We have
\begin{equation}\label{5.2} 
g_A(y)=\frac{A^2y^{|R_A|}(1+y^{|R_A|})P_0(y)}{(1-y^{|R_A|})^3}+\frac{2Ay^{|R_A|}P_1(y)}{(1-y^{|R_A|})^2} +\frac{P_2(y)}{1-y^{|R_A|}}, 
\end{equation}
where (here and in what follows, for $j\in\Z$ and $i\in R_A$, we write for brevity $i+j\in R_A$ instead of $i+j \!\pmod{A} \in R_A$)
$$
P_0(y)=\sum_{i\in R_A}\sum_{\substack{1\leq j\leq A\\i+j\,\in R_A}}y^{r(i,j)},
$$
$$
P_1(y)=\sum_{i\in R_A}\sum_{\substack{1\leq j\leq A\\i+j\,\in R_A}}jy^{r(i,j)},
$$
$$
P_2(y)=\sum_{i\in R_A}\sum_{\substack{1\leq j\leq A\\i+j\,\in R_A}}j^2y^{r(i,j)}.
$$

Let $s_1,\ldots,s_{|R_A|}$ be the distances between consecutive quadratic residues modulo $A$; here and further we consider the indices of $s_i$ modulo $|R_A|$. For $k=0,...,|R_A|-1$ we set 
\begin{equation}\label{5.3}
\a_k=\sum_{i=1}^{|R_A|}\left(\sum_{l=i}^{i+k}s_l\right)^2.
\end{equation}
In particular, $\a_0=\sum_{i=1}^{|R_A|}s_i^2=S(R_A)$ and $\a_{|R_A|-1}=A^2|R_A|$. 

Now we compute the coefficients of the polynomials $P_0$, $P_1$, and $P_2$.

\begin{lem}\label{lem5.2}
We have $P_0(y)=|R_A|\sum_{k=0}^{|R_A|-1}y^k$, $P_1(y)=A\sum_{k=0}^{|R_A|-1}(k+1)y^k$, and $P_2(y)=\sum_{k=0}^{|R_A|-1}\a_ky^k$. 
\end{lem}

\begin{proof} Recall that $r(i,j)=\#((i,i+j)\cap R_A)$ (see (\ref{3.1})). For any $s\in\{0,1,2\}$, we have
$$
P_s(y)=\sum_{k=0}^{|R_A|-1}y^k\sum_{i\in R_A}\sum_{\substack{1\leq j\leq A\\j+i\,\in R_A}}j^s1_{\#((i,i+j)\cap R_A)=k)}
=\sum_{k=0}^{|R_A|-1}y^k\sum_{i=1}^{|R_A|}\left(\sum_{l=i}^{i+k}s_l\right)^s.
$$
Now if $s=0$, then we obtain $P_0(y)=|R_A|\sum_{k=0}^{|R_A|-1}y^k$. If $s=1$, then $\sum_{i=1}^{|R_A|}\sum_{l=i}^{i+k}s_l=\sum_{l=1}^{|R_A|}s_l(k+1)=(k+1)A$, and thus $P_1(y)=A\sum_{k=0}^{|R_A|-1}(k+1)y^k$. Finally, if $s=2$, then $P_2(y)=\sum_{k=0}^{|R_A|-1}\a_ky^k$ by the definition (\ref{5.3}) of the numbers $\a_k$. This completes the proof.
\end{proof}

Lemma \ref{lem5.2} gives us    
$$(1-y)P_0(y)=|R_A|(1-y^{|R_A|})
$$
and
$$
(1-y)^2P_1(y)=A(1-y)\left(\sum_{k=0}^{|R_A|-1}y^k-|R_A|y^{|R_A|}\right)=
A(1-y^{|R_A|})-A|R_A|(1-y)y^{|R_A|}.
$$
Hence
\begin{equation}\label{5.4}
\frac{A^2y^{|R_A|}(1+y^{|R_A|})(1-y)^3P_0(y)}{(1-y^{|R_A|})^3}=
\frac{A^2|R_A|y^{|R_A|}(1+y^{|R_A|})(1-y)^2}{(1-y^{|R_A|})^2} 
\end{equation}
and 
\begin{equation}\label{5.5} 
\frac{2Ay^{|R_A|}(1-y)^3P_1(y)}{(1-y^{|R_A|})^2}=
\frac{2A^2y^{|R_A|}(1-y)}{1-y^{|R_A|}}-
\frac{2A^2|R_A|y^{2|R_A|}(1-y)^2}{(1-y^{|R_A|})^2}. 
\end{equation}
Now we compute the function $f_A(y)=(1-y)^3g_A(y)$. From (\ref{5.2}), (\ref{5.4}), and  (\ref{5.5}) we have
\begin{multline}\label{5.6} 
f_A(y)=\frac{A^2|R_A|(1-y)^2y^{|R_A|}}{1-y^{|R_A|}}+
\frac{2A^2(1-y)y^{|R_A|}}{1-y^{|R_A|}}+
\frac{(1-y)^3P_2(y)}{1-y^{|R_A|}}\\
=\frac{A^2y^{|R_A|}(2+|R_A|(1-y))+(1-y)^2P_2(y)}{1+y+\ldots+y^{|R_A|-1}}. 
\end{multline}

We continue working with the function $f_A$. Recall that we think of the indices of $s_i$ as elements of $\Z/|R_A|\Z$.

\begin{lem}\label{lem5.3} 
We have
$$
f_A(y)=\frac{F(y)}{Q(y)}, 
$$
where $Q(y)=Q_A(y)=1+y+\ldots+y^{|R_A|-1}$ and $F(y)=F_A(y)=\sum_{k=0}^{|R_A|}\b_ky^k$ is the reciprocal polynomial with the coefficients $\b_0=\b_{|R_A|}=\sum_is_i^2$ and $\b_k=2\sum_{i=1}^{|R_A|}s_is_{i+k}$ 
for $0<k<|R_A|$. 

\end{lem}

\begin{proof} From (\ref{5.6}) and Lemma \ref{lem5.2} we see that
\begin{equation*}
F(y)=
(1-2y+y^2)\sum_{k=0}^{|R_A|-1}\a_ky^k+A^2(|R_A|+2)y^{|R_A|}-A^2|R_A|y^{|R_A|+1}.
\end{equation*}
Now we write $F(y)=\sum_{k=0}^{|R_A|+1}\g_ky^k$ and compute $\g_k$. Firstly, by the definition (\ref{5.3}) of the numbers $\a_k$,
$$
\g_0=\a_0=\sum_{i=1}^{|R_A|}s_i^2,
$$
and
$$\g_1=\a_1-2\a_0=\sum_{i=1}^{|R_A|}(s_i+s_{i+1})^2-2\sum_{i=1}^{|R_A|}s_i^2=2\sum_{i}s_is_{i+1}.
$$
Further, we observe that 
\begin{equation*} \a_k=\sum_i(s_i+\ldots+s_{i+k})^2=
(k+1)\sum_is_i^2+2\sum_{l=1}^k(k+1-l)\sum_{i}s_is_{i+l}.
\end{equation*} 
Therefore, for $2\leq k\leq |R_A|-1$,   
$$
\g_k=\a_{k-2}-2\a_{k-1}+\a_k=2\sum_is_is_{i+k},
$$
as desired. It remains to calculate $\g_{|R_A|}$ and $\g_{|R_A|+1}$. Recall that $\a_{|R_A|-1}=A^2|R_A|$. Hence
$$
\g_{|R_A|+1}=\a_{|R_A|-1}-A^2|R_A|=0
$$   
and
$$
\g_{|R_A|}=-2\a_{|R_A|-1}+\a_{|R_A|-2}+A^2(|R_A|+2)=\a_{|R_A|-2}+2A^2-A^2|R_A|.
$$
Finally,
\begin{multline*}
\a_{|R_A|-2}=\sum_{i}\left(\sum_{l=i}^{i+|R_A|-2}s_l\right)^2=\sum_{i=1}^{|R_A|}(A-s_i)^2\\
=A^2|R_A|-2A\sum_is_i+\sum_is_i^2=A^2|R_A|-2A^2+\a_0,
\end{multline*} 
and therefore $\g_{|R_A|}=\g_0=\sum_is_i^2$. Thus $F(y)=\sum_{k=0}^{|R_A|}\b_ky^k$, as desired.
\end{proof}

Since $f_A(y)=(1-y)^3g_A(y)$, we have $g_A(0.5)=8f_A(0.5)$ and $g_A(y_t)=2^{3t}f_A(y_t)$ for $y_t=1-2^{-t}$. Now Theorem \ref{th1.4} follows immediately from Corollary \ref{cor4.6}; we can also rewrite Corollary \ref{cor4.7} as  
\begin{equation}\label{5.7}
S(R_M)=m2^tf_A(y_t)+O(M(\log M)^{-1})
\end{equation}
for $M\in\O$. To complete the proof of Theorem \ref{th1.1}, we are going to expand the function $f_A$ in a Taylor series at the point $y=1$.

We see from Lemma \ref{lem5.3} that $f_A$ obeys the functional equation
$$
f_A(y)=yf_A\left(\frac1y\right).
$$
We then obtain
\begin{equation}\label{5.8}
f'_A\left(\frac1y\right)=f_A(y)-yf_A'(y)
\end{equation}
and
\begin{equation}\label{5.9}
f_A''\left(\frac1y\right)=y^3f''_A(y).
\end{equation}
Also from (\ref{5.6}) we have
\begin{equation}\label{5.10}
f_A(1)=\frac{2A^2}{|R_A|},
\end{equation}
and now (\ref{5.8}) immediately implies $f'_A(1)=0.5f_A(1)=A^2|R_A|^{-1}$. Then by the Taylor expansion 
\begin{equation}\label{5.11}
m2^tf_A(y_t)=m2^{t+1}A^2|R_A|^{-1}-mA^2|R_A|^{-1}+\frac12f''_A(\theta_{A,t})m2^{-t}
\end{equation}
for some $\theta_{A,t}\in(1-2^{-t},1)$, and it remains to estimate the second derivative of $f_A$. 

\begin{lem}\label{lem5.4}
Suppose that $|R_A|\leq 2^{t}$. Then
$$
f''_A(y)\ll A^2|R_A|
$$
uniformly for $y\in(y_t,1)$.
\end{lem}

\begin{proof} Due to (\ref{5.9}), it suffices to prove the lemma for $1\leq y\leq 1+O(2^{-t})$ instead of $y\in(y_t,1)$. Firstly,
\begin{equation}\label{5.12}
\left(\frac{F(y)}{Q(y)}\right)''=\frac{F''(y)Q^2(y)-2F'(y)Q'(y)Q(y)-F(y)Q''(y)Q(y)+2F(y)Q'^2(y)}{Q^3(y)}.  
\end{equation}
Further, for $1\leq y\leq 1+O(2^{-t})$ by the assumption on $|R_A|$ we have $y^k \asymp 1$ uniformly for $0\leq k\leq |R_A|$ and thus
$$
Q(y)\asymp |R_A|, \qquad Q'(y)\asymp |R_A|^2, \qquad Q''(y) \asymp |R_A|^3.
$$ 
Taking into account (\ref{5.10}), we find
$$
F(y)\ll \sum_{k=0}^{|R_A|}\b_k=F(1)=f_A(1)Q(1)=2A^2,
$$
$$
F'(y)=\sum_{k=1}^{|R_A|}k\b_ky^{k-1}\ll A^2|R_A|, \qquad F''(y)\ll A^2|R_A|^2. 
$$
Thus for $1\leq y\leq 1+O(2^{-t})$ we have from (\ref{5.12}) 
$$
f_A''(y) \ll \frac{A^2|R_A|^4}{|R_A|^3}=A^2|R_A|,
$$
and the claim follows. 
\end{proof}

Now combining (\ref{5.7}), (\ref{5.11}), and the bound from Lemma \ref{lem5.4}, we get 
\begin{equation}\label{5.13}
S(R_M) = M2^{t+1}A|R_A|^{-1}-MA|R_A|^{-1} + O(M(\log M)^{-1}+MA|R_A|2^{-t}).
\end{equation}
Recalling that $M\in\O$ implies $A\leq (\log M)^{0.1}$ and $2^{-t}\leq2^{-(0.9+o(1))\log\log M}$, we get 
\begin{equation}\label{5.14}
A|R_A|2^{-t}\leq (\log M)^{-0.4}. 
\end{equation}
Further, since $\prod_{i=1}^t(1+p_i^{-1})\leq\exp(\sum_{i=1}^tp_i^{-1})=1+O(tp_1^{-1})$,
we have
\begin{equation*}
|R_M|=|R_A|\prod_{i=1}^t\frac{p_i+1}{2}=|R_A|m2^{-t}(1+O(tp_1^{-1})), 
\end{equation*}
and hence
\begin{equation*}
\frac{M}{|R_M|}=A|R_A|^{-1}2^t(1+O(tp_1^{-1}))=A|R_A|^{-1}2^t+O((\log M)^{1-C_0})
\end{equation*}
and
\begin{equation}\label{5.15}
\frac{2M^2}{|R_M|}=M2^{t+1}A|R_A|^{-1}+O(M(\log M)^{1-C_0}).
\end{equation}
Putting (\ref{5.13}), (\ref{5.14}), and (\ref{5.15}) together and recalling that $C_0>2$, we conclude the proof of Theorem \ref{th1.1}.

\section{Computing the density of the set $\O$} \label{Sec6}

Let $X>0$ be a large number and $\O(X)=\{M\leq X: M\in\O \}$. We first prove the lower bound for $|\O(X)|$. To do this, we consider the set $\O_1(X)$ of sufficiently large positive integers $X^{1/2}<M\leq X$ such that

\begin{itemize}
	
\item[(i)] $M$ is square-free;
	
\item[(ii)] $0.9\log\log X < \o(M) < 1.05\log\log X$;	
	
\item[(iii)] $M$ has the representation $M=Am$, where $P^-(m)>(\log X)^{C_0}$ and \\ $3\leq A\leq0.5(\log X)^{0.1}$.
\end{itemize}

Clearly, $\O_1(X)\subseteq \O(X)$. We are going to show that $|\O_1(X)|= (\eta+o(1))X$, where $\eta=e^{-\g}\frac{3}{5\pi^2C_0}$ and $\gamma$ is the Euler constant. It is well-known that the number of $M\leq X$ violating (ii) is $O(X(\log X)^{-c})$ for some absolute constant $c>0$ (see, for example, \cite[Exercise 04]{HT}). Further, let $z=(\log X)^{C_0}$ and $V(X,z)$ be the set of square-free $m\leq X$ such that $P^-(m)>z$. We are aiming to prove that
\begin{equation}\label{6.1}
|V(X,z)|=\frac{e^{-\gamma}X}{\log z}+O\left(\frac{X}{(\log z)^2}\right). 
\end{equation}
Set $P(z)=\prod_{p\leq z}p$. Obviously, $V(X) = V_1\setminus V_2$, where 
$$
V_1=\{m\leq X: (m,P(z))=1\}
$$
and
$$
V_2=\{m\leq X: \exists p>z \mbox{ with } p^2|m \}.
$$
Standard sieve methods (see, for example, \cite[Exercise 9.1.9 and Theorem 9.1.3]{Mur}) yield that, for our choice of $z$, 
\begin{equation*} 
|V_1| = \frac{e^{-\g}X}{\log z}+O\left(\frac{X}{(\log z)^2}\right); 
\end{equation*}
also we have
\begin{equation*}
|V_2|\leq \sum_{p>z}\left\lfloor\frac{X}{p^2}\right\rfloor \ll Xz^{-1}.
\end{equation*}
These two estimates imply (\ref{6.1}). 

Now for any $A\leq 0.5(\log X)^{0.1}$ we obtain from (\ref{6.1})
$$
|V(X/A,z)|=\frac{e^{-\gamma}X}{C_0A\log\log X}+O\left(\frac{X}{A(\log\log X)^2}\right).
$$
Let $\mu$ denote the M\"obius function. It is well-known that $\sum_{A\leq y}\mu^2(A)=\frac{6}{\pi^2}y+O(y^{1/2})$; by partial summation we get 
$$
\sum_{3\leq A\leq y}\frac{\mu^2(A)}{A}=\frac{6}{\pi^2}\log y+O(1).
$$ 
Thus
\begin{multline*}
|\O_1(X)| = \sum_{3\leq A\leq 0.5(\log X)^{0.1}} \mu^2(A)|V(X/A,z)|+O\left(\frac{X}{(\log X)^c}\right)\\ 
=\frac{e^{-\g}X}{C_0\log\log X}\sum_{3\leq A\leq0.5(\log X)^{0.1}}\frac{\mu^2(A)}{A}+o(X)=(\eta+o(1))X.
\end{multline*}

To establish the upper bound on $|\O(X)|$, we consider the set $\O_2(X)$ of positive integers $X^{1/2}<M\leq X$ such that

\begin{itemize}
	
\item[(i')] $M$ is square-free;
	
\item[(ii')] $0.85\log\log X < \o(M) < 1.1\log\log X$;	
	
\item[(iii')] $M$ has the representation $M=Am$, where $P^-(m)>(0.5\log X)^{C_0}$ and \\ $3\leq A\leq(\log X)^{0.1}$.

\end{itemize}
Arguing as before, we see that $|\O_2(X)|=(\eta+o(1))X$. Finally, any $M\in\O(X)$ with $M>X^{1/2}$ belongs to $\O_2(X)$, and thus $|\O(X)|\leq(\eta+o(1))X$. This completes the proof.

\section{Proof of the identity (\ref{1.3})} \label{Sec7}

We have
\begin{equation}\label{7.1}
{M \choose k}s(k)=\sum_{|U|=k}S(U)=\frac{1}{k}\sum_{x\in\Z/M\Z}\sum_{\substack{|U|=k\\U\ni x}}S(U)=\frac{M}{k}\sum_{\substack{|U|=k\\U\ni 0}}S(U).
\end{equation}
Setting $\A_{M,k}=\left\{\overline{x}=(x_1,\ldots,x_k)\in \N^k : \sum_{i=1}^kx_i=M\right\}$; we have $|\mathcal{A}_{M,k}|={M-1 \choose k-1}$. Note that a set $U=\{0=u_1<u_2<\ldots<u_k<M\}\subseteq \Z/M\Z$ containing $0$ is uniquely determined by a $k$-tuple $\overline{x}\in \A_{M,k}$ of the distances between its elements. Then by (\ref{7.1}) and the symmetry
$$
{M-1 \choose k-1}s(k)=\sum_{\substack{|U|=k\\U\ni 0}}S(U) = \sum_{\overline{x}\in\A_{M,k}}\sum_{i=1}^k{x_i^2}=k\sum_{\overline{x}\in\A_{M,k}}{x_1^2}.
$$
Further, the number of $\overline{x}\in\A_{M,k}$ with $x_1=a$ (here $a=1,...,M-k+1$) is equal to $|\A_{M-a,k-1}|={M-a-1 \choose k-2}$. Thus
\begin{equation}\label{s(k)2}
{M-1 \choose k-1}s(k)=k\sum_{\nu=1}^{M-k+1}a^2{M-1-a \choose k-2}.
\end{equation}
It remains to compute the latter sum. By induction on $v$ it is easy to prove the identities (first (\ref{7.3}), then using it, (\ref{7.4}), and then (\ref{7.5}))
\begin{equation}\label{7.3}
\sum_{a=0}^v{u+v-a \choose u} = {u+v+1 \choose u+1},
\end{equation}
\begin{equation}\label{7.4}
\sum_{a=1}^va{u+v-a \choose u} = {u+v+1 \choose u+2},
\end{equation}
\begin{equation}\label{7.5}
\sum_{a=1}^va^2{u+v-a \choose u} = 2{u+v+1 \choose u+3}+{u+v+1 \choose u+2}.
\end{equation}
Using (\ref{7.5}) for $u=k-2$ and $v=M-k+1$, we get from (\ref{s(k)2}) 
$$
s(k)=k\frac{2{M \choose k+1}+{M \choose k}}{{M-1 \choose k-1}}=M\frac{2M-k+1}{k+1}\,,
$$
as desired.

\smallskip 

\textbf{Acknowledgements.} This work was supported by the Russian Science Foundation under grant no.19-11-00001, https://rscf.ru/en/project/19-11-00001. The author would like to thank Sergei Konyagin and Yuri Malykhin for helpful discussions, and Pavel Sobyanin and Konstantin Olmezov for computational work concerned with the values $f_A(0.5)$ for small $A$. The author is also grateful to the anonymous referees for many suggestions which improved the quality of the paper.

\end{document}